\documentclass[11pt,margin=1in,lmargin=1.25in, reqno]{amsart}
\usepackage{amsfonts, amsmath, amssymb, color}
\usepackage{listings}
\usepackage[pagewise]{lineno}
\usepackage{amsthm}
\usepackage{booktabs}
\usepackage{float}
\usepackage{hhline}
\usepackage{lipsum}
\usepackage{lscape}
\usepackage{subfig}
\usepackage{textcomp}
\usepackage{tikz}
\usetikzlibrary{decorations.pathmorphing,shapes,arrows,positioning}
\usepackage{xcolor}
\usepackage{young}
\usepackage[enableskew,vcentermath]{youngtab}
\Ylinethick{0.8pt}
\Yautoscale{0}
\Yboxdim{15pt}
\usepackage{ytableau}
\allowdisplaybreaks[4]
\theoremstyle{plain}
\newtheorem{thm}{Theorem}[section]
\newtheorem{lem}[thm]{Lemma}
\newtheorem{prop}[thm]{Proposition}
\newtheorem{cor}[thm]{Corollary}

\theoremstyle{definition}

\newtheorem{exmp}[thm]{Example}

\usepackage[colorlinks=true, linkcolor=blue, citecolor=magenta, urlcolor=blue, 
]{hyperref}   

\makeatletter

\newcommand{\Rmnum}[1]{\expandafter\@slowromancap\romannumeral #1@}
\makeatother
\newcommand{\la}{\lambda}

\numberwithin{equation}{section} \errorcontextlines=0

\begin{document}
\title{Irreducible characters of the generalized symmetric group}
\author{Huimin Gao}
\address{School of Mathematics, South China University of Technology, Guangzhou, Guangdong 510640, China}
\email{13144508207@163.com}
\author{Naihuan Jing}
\address{Department of Mathematics, North Carolina State University, Raleigh, NC 27695, USA}
\email{jing@ncsu.edu}

\maketitle

\begin{abstract}
The paper studies how to compute irreducible characters of the
generalized symmetric group $C_k\wr{S}_n$ by iterative algorithms. After reproving the Ariki-Koike version of
the Murnaghan-Nakayama rule by vertex algebraic methods, we formulate a new iterative formula for characters of the generalized symmetric group.
As applications, we find a numerical relation between the character values of $C_k\wr S_n$ and
modular characters of $S_{kn}$.
\end{abstract}

\section{Introduction}

The character theory of the symmetric group $S_n$ plays an important role in representation theory \cite{JK}. It is well known that the irreducible characters $\chi^\lambda$ of $S_n$ are indexed by partitions $\lambda$ of $n$,  
and explicit character values are given by
the celebrated Frobenius formula, which expresses the character values as the transition coefficients between the Schur symmetric functions and the power-sum symmetric functions \cite{JK, Mac}.
The character formula can also be formulated using a vertex algebraic method \cite{J1}, where
the irreducible character values of $S_n$ are expressed by matrix coefficients of Bernstein vertex operators \cite{Zel}
and products of Heisenberg operators.

The theory of irreducible characters of the wreath product $G\wr S_n$, where $G$ is any finite group, was given by Specht \cite{S}. Under the Frobenius-type charactersitic, the irreducible characters $\chi^{\boldsymbol{\lambda}}$ of $G\wr S_n$
correspond to the wreath product Schur symmetric functions \cite{Sta} labelled by partitions $\boldsymbol{\lambda}$ colored by the irreducible characters of $G$. The character values are then given by the transition coefficients between the wreath product Schur symmetric functions and
power-sum symmetric functions of partitions colored by the conjugacy classes of $G$.

Zelevinsky also studied the representations of $G\wr S_n$ using the language of PSH-algebras \cite{Zel}. The Grothendieck group of the category of finite-dimensional complex representations possesses an additional structure of Hopf algebra obeying positivity and self-adjointness axioms.

In \cite{FJW, IJS} Frenkel, Wang and the second named author reformulated the Specht
character theory using the vertex algebraic method in the context of the McKay correspondence. 
The character values are given by matrix coefficients of vertex operators of the form
$$\langle \prod p_{\rho_i}(c), \prod X_{\gamma}(\lambda_i)\rangle,$$
where $X_{\gamma}(\lambda_i)$ are the components of some vertex operators indexed by irreducible character $\lambda_i$ of $G$ and $p_{\rho_i}(c)$ are power-sum symmetric functions colored by the conjugacy classes of $G$.

It is generally believed that the character theory of {\it the generalized symmetric group} $C_p\wr S_n$, the wreath product of the cyclic group $C_p$ and $S_n$, is closely related to the modular character theory of the symmetric group $S_{n}$ over characteristic $p$ (cf. \cite{KOR}). All Weyl groups of classical type are essentially generalized symmetric groups. 
If $k=1$, the generalized symmetric group reduces to $S_n$; if $k=2$, it is specially called the hyperoctahedral group, which is isomorphic to the Weyl group in type $B$. A character identity relating irreducible character values of the hyperocta-hedral group $C_2\wr S_n$ and those of the symmetric group $S_{2n}$ was recently found \cite{LP, AR}. It is natural to study how to effectively
compute character values of the generalized symmetric group $C_k\wr S_n$. Pfeiffer has given programs to compute
the character tables of the Weyl groups in GAP \cite{P}.

The classical Murnaghan-Nakayama rule is a combinatorial rule for computing the irreducible character $\chi^\lambda$ on conjugacy class $\rho$ of the symmetric group $S_n$ \cite{Mur, Nak}. Recall that a {\it partition} $\lambda$ of $n$ is a decreasing sequence of positive integers $\lambda_i$ whose total sum is $n$, denoted by $\lambda \vdash n$. The number of $\lambda_i$ is the length of $\la$, denoted by $l(\la)$. If the parts $\lambda_i$ are not ordered, then we call $\lambda$ a {\it composition} of $n$.


There are several generalizations of the classical Murnaghan-Nakayama rule. Osima \cite{O} had given a Murnaghan-Nakayama rule in terms of skew representations of the generalized symmetric groups, and Stembridge \cite{St} rediscovered the rule and studied associated combinatorial formulas. This version of the Murnaghan-Nakayama rule requires information of restricting
an irreducible representation of $C_k\wr S_n$ to $C_k\wr S_{n-m}$, so it needs extra work for computer programming. 
Finally Ariki and Koike \cite{AK} found the Murnaghan-Nakayama rule for
the generalized symmetric group that naturally generalizes the classical one and independent of a priori information of the restriction functor. There are also generalizations of the Murnaghan-Nakayama rule for other related structures \cite{
MRW, BSZ, EPW, HR, JL}. 

In this paper, we study the Murnaghan-Nakayama rule for the generalized symmetric groups $C_k\wr S_n$ aiming to compute the characters effectively. 
We first reformulate the Ariki-Koike version of the
Murnaghan-Nakayama rule using 
the vertex algebraic method. Our treatment relies upon the vertex operator realization of Schur functions \cite{J2, J3} 
to derive the iterative formula for the irreducible characters, which is perhaps a quicker derivation of the rule. Another benefit of the vertex algebraic method is that it also gives a {\it new 
iterative rule} or the dual rule for
the generalized symmetric group by reducing the partition of the irreducible representation by rows. 

Our reformulation of the Murnaghan-Nakayama rule is the following result, which is obtained by splitting the operators corresponding to conjugacy classes. This version gives an iterative formula to compute the irreducible characters that may simplify computational complexity.
\begin{thm}
		Given colored partitions $\boldsymbol{\lambda}=(\lambda^{(0)},\lambda^{(1)},\cdots,\lambda^{(k-1)})$ and $\boldsymbol{\rho}=(\rho^{(0)},\rho^{(1)},\cdots,\rho^{(k-1)})$ of $n$,
 where $\rho^{(s)}=(\rho^{(s)}_1,\cdots,\rho^{(s)}_m,\cdots,\rho^{(s)}_{l(s)})$. Then the value of the irreducible character $\chi^{\boldsymbol{\lambda}}$ of $C_k\wr S_n$ at the class $\boldsymbol{\rho}$ is given by
	\begin{align}
		\chi^{\boldsymbol{\lambda}}_{\boldsymbol{\rho}}=\sum_{j=0}^{k-1}\sum_{\boldsymbol{\xi_j}}(-1)^{\mathrm{ht}(\boldsymbol{\lambda_j})}\omega^{-sj}\chi^{\boldsymbol{\lambda}\backslash\boldsymbol{\xi_j}}_{\boldsymbol{\rho}\setminus\rho^{(s)}_m},
	\end{align}
where ${\boldsymbol{\xi_j}}$ runs through all colored $\rho^{(s)}_m$-rim hooks contained in ${\boldsymbol{\lambda}}$ that are supported at the j-th constituent.
\end{thm}

Moreover, we can also break down the irreducible characters into lower rank ones using a dual procedure, which leads to a new iterative rule for the generalized symmetric group. 
Our rule at $k=1$ seems to be a
new combinatorial formula for the irreducible character values of the symmetric group $S_n$ as well.
\begin{thm} (New iterative rule)
	Given colored partitions $\boldsymbol{\lambda}=(\lambda^{(0)},\lambda^{(1)},\\
 \cdots,\lambda^{(k-1)})$ and $\boldsymbol{\rho}=(\rho^{(0)},\rho^{(1)},\cdots,\rho^{(k-1)})$ of $n$. 
 For any fixed color $j$, the irreducible character value of $C_k\wr S_n$ is given by
    \begin{align}
    	\chi^{\boldsymbol{\lambda}}_{\boldsymbol{\rho}}=\sum_{\substack{\boldsymbol{\mu}\vartriangleleft\boldsymbol{\rho},\|\boldsymbol{\mu}\|\geq\lambda^{(j)}_1\\\boldsymbol{\tau}\vdash \|\boldsymbol{\mu}\|-\lambda^{(j)}_1}}\omega^{(\eta(\boldsymbol{\tau})-\eta(\boldsymbol{\mu}))j}\frac{{(-1)}^{l(\boldsymbol{\tau})}}{k^{l(\boldsymbol{\tau})}z_{\tau}}\chi^{\boldsymbol{\lambda}\setminus \lambda^{(j)}_1}_{(\boldsymbol{\rho}\backslash\boldsymbol{\mu})\cup\boldsymbol{\tau}}
    \end{align}
summed over colored partitions $\boldsymbol{\mu}$ and $\boldsymbol{\tau}$ such that $\boldsymbol{\mu}\vartriangleleft\boldsymbol{\rho}$ with weight bigger than $\lambda^{(j)}_1$ and $\boldsymbol{\tau}$ of weight $\|\boldsymbol{\mu}\|-\lambda^{(j)}_1$. Here $z_{\tau}$ is defined below \eqref{e:innerprod}, $\tau$ is the rearrangement of $\boldsymbol{\tau}$, and $\eta(\boldsymbol{\tau})$, $\eta(\boldsymbol{\mu})$ are the weighted lengths
$\sum_iil(\tau^{(i)})$, $\sum_i il(\mu^{(i)})$ respectively.
\end{thm}

We point it out that the dual Murnaghan-Nakayama rule for
the generalized symmetric group gives a different iteration
procedure. The former one given by Ariki and Koike simplifies the indexing colored-partition by removing all rim-hooks in each constituent, while our new rule simply removes
the largest part in a fixed constituent. This reveals a different relation between representations of the generalized symmetric groups.

We also show how to use the Sagemath program \cite{Zim} to implement the algorithm. In theory, any irreducible character value of the generalized symmetric group can be iterated once according to our program pending on computer's CPU. The source code for this program can be found in the appendix. As applications, we also find a numerical relation between the irreducible character values
of $C_k \wr S_n$ and those of modular $S_{kn}$ mod $k$.

The paper is organized into three parts. In section 2, we recall the basic notions about generalized symmetric groups $C_k\wr S_n$ and the construction of the irreducible character value \cite{FJW}. In section 3, we derive the analog of the Murnaghan-Nakayama rule and another iterative formula for the generalized symmetric groups, using the technique of vertex operators. The dual Maunaghan-Nakayama rule for the symmetric group is also obtained as a consequence. In section 4, we discuss some examples and properties of irreducible character values based on our main results. In the end, we formulate a numerical relation between
the irreducible character values of $C_k\wr S_n$ and those of modular $S_{kn}$. The character tables of $C_3\wr S_1$, $C_3\wr S_2$, $C_3\wr S_3$ are attached in the end.

\section{Irreducible characters of $C_k\wr S_n$}

The generalized symmetric group $C_k\wr S_n$ is the wreath product of the cyclic group $C_k$ of order $k$ with the symmetric group $S_n$ of $n$ elements,
i.e. the semidirect product $C_k^{n}\rtimes S_n$, where the symmetric group $S_n$ acts on the direct product $C_k^{n}=C_k\times \cdots \times C_k$ by permutating the factors. The multiplication is given by
\begin{align}
	(g;\sigma)\cdot (h;\tau)=(g\sigma(h);\sigma\tau),
\end{align}
where $g,h\in C_k^{n}$ and $\sigma,\tau\in S_n$.

Since $C_k$ is an abelian group of order $k$, there are exactly $k$ conjugacy classes, given by
$\{1\},\{c\},\cdots,\{c^{k-1}\}$, where $C_k=\langle c \rangle$.
The order of the centralizer of each conjugacy class is $k$. Its irreducible characters are given by
$\gamma_i$, $i=0, \ldots, k-1$, where
$\gamma_i(c^j)=\omega^{ij}$, and $\omega=e^{2\pi i/k}$ is the $k$-th primitive root of unity.
The space of class functions on $C_k$ is given by
\begin{align*}
	R(C_k)=\bigoplus_{i=0}^{k-1}\mathbb{C}\gamma_i.
\end{align*}

The conjugacy classes of $C_k\wr S_n$ are parametrized by colored partitions. The color set is $I=\{0,1,\cdots,k-1\}$, the indexing set of
conjugacy classes. The conjugacy class of an element $x=(g;\sigma)\in C_k\wr S_n$ corresponds to {\it the $I$-colored partition} $\boldsymbol{\lambda}=(\lambda^{(i)})$, which consists of $|I|$ partitions $\lambda^{(i)}$ such that $\Arrowvert\boldsymbol{\lambda}\Arrowvert=\sum_{i=0}^{k-1}|\lambda^{(i)}|=n$, the $i$th partition $\lambda^{(i)}$  given by
\begin{align*}
	\lambda^{(i)}=(1^{m_1(c^{i})}2^{m_2(c^{i})}\cdots),
\end{align*}
where $m_j(c^i)$ equals to the number of $j$-cycles in $\sigma$ whose cycle product is $c^i$. We also say that the colored partition $\boldsymbol{\lambda}=(\lambda^{(i)})$ is supported at its constituent subpartitions $\lambda^{(i)}$. In particular, a colored partition $\boldsymbol{\lambda}$ supported at only one color has one nontrivial constituent subpartition, and all other constituents are empty. 

Given a partition $\la=(\la_1, \la_2, \ldots, \la_l)$, one can paint the parts with the colors $\{0, 1, \ldots, k-1\}$, to get $k^{l(\la)}$ colored partitions $\boldsymbol{\la}$, which are referred as the colored partitions arising from $\la$. On the other hand, all these $k^{l(\la)}$ colored partitions become to the partition $\la$ by erasing their colors.
If $\boldsymbol{\lambda}=(\lambda^{(i)})$ is a colored partition, we define the {\it weighted length} $\eta(\boldsymbol{\lambda})$ as $\sum_i il(\lambda^{(i)})$.

We also call $\boldsymbol{\lambda}$ the
type of the conjugacy class. It is known that two elements are conjugate in $C_k\wr S_n$ if and only if they have the same type.

A {\it hook} is a special partition of the form $(a, 1^b)$, where $a-1$ is the length of its arm and $b$ is the length of its leg. The hook length is defined as $a+b$. For a general partition $\lambda$, its hook length at $x=(i, j)\in \lambda$, of the Young diagram, is defined by
$$h(x)=\lambda_{i}+\lambda_j'-i-j+1,$$
and the hook length of $\lambda$ is defined as the product of $h(x)$ for $x\in \lambda$, where $\lambda'$ is the conjugate of $\lambda$.

\begin{exmp}
	In $C_3\wr S_7$,
	\begin{align*}
		x=((1,1,1,\omega^1,\omega^1,\omega^2,1);(1,2,3)(4,5)(6,7));\\
		y=((1,\omega^2,\omega^1,1,1,1,\omega^1);(1,4,5)(2,6)(3,7))
	\end{align*}
	 are conjuagate, since they have the same type $((3),(\phi),(2,2))$.
\end{exmp}

Irreducible characters of the generalized symmetric group $C_k\wr S_n$ were determined by Specht \cite{S}. Frenkel, Wang, and the second author have
given a vertex algebraic method to express the irreducible character values by matrix coefficients of certain vertex operators
\cite{FJW}, which we now recall in the following.

Let $\{a_n(\gamma_i)|n\in \mathbb{Z},i\in I\}\cup\{c\}$ be the set of generators of the Heisenberg algebra $\mathcal{H}$ with defining relations
\begin{align}
	[a_n(\gamma_i),a_m(\gamma_j)]=n\delta_{n,-m}\delta_{i,j}c,
\end{align}	
\begin{align}
    [a_n(\gamma_i),c]=0.
\end{align}

Consider the Fock space $V=\mathrm{Sym}(a_{-n}(\gamma_i)'s)$, the polynomial algebra on the variables $a_{-n}(\gamma_i)$,$n\in \mathbb{N}$ and $i\in I$. The algebra $\mathcal{H}$ acts on $V$ by the following rule:
\begin{align}
	&a_{-n}(\gamma_i).v=a_{-n}(\gamma_i)v,\\
	&c.v=v, \\ \notag
    &a_n(\gamma_i).a_{-n_1}(\alpha_1)a_{-n_2}(\alpha_2)\cdots a_{-n_k}(\alpha_k)\\ \label{e:2.6}
	&=\sum_{j=1}^{k}\delta_{n,n_j}\langle \gamma_i,\alpha_j\rangle a_{-n_1}(\alpha_1)a_{-n_2}(\alpha_2)\cdots a_{-n_k}(\alpha_k)
\end{align}
where $\alpha_j\in R(C_k)$, $\langle \gamma_i,\gamma_j\rangle =\delta_{i,j}$ and $\langle \gamma_i,\alpha_j\rangle$ is extended the whole space by linearity.

For convenience, we denote $a_n(\gamma_i):=a_{in}$, where the first index $i$ stands for $\gamma_i$. The space $V$ is equipped with a bilinear form given by
\begin{align}
	\langle 1,1\rangle =1, \quad a_{in}^*=a_{-in},
\end{align}	
where $a_{in}^*$ denotes the adjoint of $a_{in}$.
We denote by $A_i(z)$ and $A_i^*(z)$  the generating series:
\begin{align}
	A_i(z)=\sum_{n\in \mathbb{N}}a_{-in}z^n, \quad A_i^*(z)=\sum_{n\in \mathbb{N}}a_{in}z^{-n}.
\end{align}

Moreover, for a partition $\lambda=(\lambda_1,\lambda_2,\cdots \lambda_l)$, we define
\begin{align}\label{e:2.9}
	a_{-i\lambda}=a_{-i\lambda_1}a_{-i\lambda_2}\cdots a_{-i\lambda_l},
\end{align}
which form an orthogonal basis:
\begin{align}\label{e:innerprod}
	\langle a_{-i\lambda}, a_{-j\mu}\rangle=\delta_{i,j}\delta_{\lambda \mu}z_\lambda,
\end{align}
where $z_\lambda=\prod_i i^{m_i}m_i!$ for $(1^{m_1}2^{m_2}\cdots)$, the exponential notation of $\lambda$.

In fact, $\mathcal{H}$ has another set of generators consisting of the Fourier transform of the first set:
$\{\widetilde{a}_{in}|n\in \mathbb{Z},i\in I\}\cup\{c\}$, where
\begin{align}\label{e:2.11}
	\widetilde{a}_{in}=\sum_{j=0}^{k-1}\omega^{-ij}a_{jn}.
\end{align}	
Under the inverse Fourier transform, we have that
\begin{align}
	a_{in}=\frac{1}{k}\sum_{j=0}^{k-1}\omega^{ij}\widetilde{a}_{jn}.
\end{align}

Similarly, for a partition $\lambda=(\lambda_1, \lambda_2, \ldots, \lambda_l)$ we denote
\begin{align}
	\widetilde{a}_{-i\lambda}=\widetilde{a}_{-i\lambda_1}\widetilde{a}_{-i\lambda_2}\cdots \widetilde{a}_{-i\lambda_l}.
\end{align}

For a colored partition $\boldsymbol{\lambda}=(\lambda^{(0)},\lambda^{(1)},\cdots,\lambda^{(k-1)})$, let
\begin{align}
	\widetilde{a}_{-\boldsymbol{\lambda}}=\widetilde{a}_{-0\lambda^{(0)}}\widetilde{a}_{-1\lambda^{(1)}}\cdots \widetilde{a}_{-(k-1)\lambda^{(k-1)}}.
\end{align}
Then $ \widetilde{a}_{-\boldsymbol{\lambda}}$ form another basis of $\mathcal{H}$. 

We introduce the vertex operator $X_i(z)$ and its adjoint $X_i^*(z)$ as the linear maps: $V\longrightarrow V[[z, z^{-1}]]$ given by
\begin{align}
	X_i(z)&=\mbox{exp} \left( \sum\limits_{n\geq 1} \frac{1}{n}a_{-in}z^{n} \right) \mbox{exp} \left( -\sum \limits_{n\geq 1} \frac{1}{n}a_{in}z^{-n} \right)\\ \notag
	&=\sum_{n\in\mathbb Z}X_{in}z^n,\\
	X_i^*(z)&=\mbox{exp} \left( -\sum\limits_{n\geq 1} \frac{1}{n}a_{-in}z^{n} \right) \mbox{exp} \left( \sum \limits_{n\geq 1} \frac{1}{n}a_{in}z^{-n} \right)\\ \notag
	&=\sum_{n\in\mathbb Z}X_{in}^*z^{-n}.
\end{align}

In fact, for a partition $\lambda=(\lambda_1,\lambda_2,\cdots, \lambda_l)$, it is known that
\begin{align}
	s_{i\lambda}=X_{i\lambda_1}X_{i\lambda_2}\cdots X_{i\lambda_l}.1
\end{align}	
is the Schur function associated to the parition $\lambda$ in the variables $a_{-in}$ (as a power-sum symmetric function) \cite{J2}. 
Then we have
\begin{align}
	\langle s_{i\lambda}, s_{i\mu}\rangle=\delta_{\lambda \mu}.
\end{align}

For a colored partition $\boldsymbol{\lambda}=(\lambda^{(0)},\lambda^{(1)},\cdots,\lambda^{(k-1)})$, we denote
	\begin{align}
		s_{\boldsymbol{\lambda}}=s_{0\lambda^{(0)}}s_{1\lambda^{(1)}}\cdots s_{(k-1)\lambda^{({k-1})}}.
\end{align}	
We recall the following vertex algebraic formulation of Specht characters:
\begin{thm} \cite{FJW} \label{t:FJW}
	Given colored partitions $\boldsymbol{\lambda}=(\lambda^{(0)},\lambda^{(1)},\cdots,\lambda^{(k-1)})$ and $\boldsymbol{\rho}=(\rho^{(0)},\rho^{(1)},\cdots,\rho^{(k-1)})$ of $n$, the matrix coefficient
	\begin{align*}
		\langle s_{\boldsymbol{\lambda}}, \widetilde{a}_{-\boldsymbol{\rho}} \rangle
	\end{align*}
is equal to the value of the irreducible character $\chi^{\boldsymbol{\lambda}}$ at the conjugacy class of type $\boldsymbol{\rho}$ in $C_k\wr S_n$.
\end{thm}

\section{Two Iterative Formulas}

Using techniques of vertex operators, we have the following commutation relations for the vertex operators (cf. \cite{J3}):
\begin{align}\label{e:com1a}
	zX_i(z)X_i(w)&+wX_i(w)X_i(z)=0, \\ \label{e:com2a}
	A_i^*(z)X_i(w)&=X_i(w)A_i^*(z)+\frac{w}{z-w}X_i(w),\\
 \label{e:com3a}
	X_i^*(z)A_i(w)&=A_i(w)X_i^*(z)+\frac{w}{z-w}X_i^*(z).
\end{align}
Taking the coefficients of $z^mw^n$, we derive the following commutation relations among the components.
\begin{prop}\label{P:3.1}
	For $m,n\in \mathbb{N}$,$i\in I$, we have
	\begin{align}\label{e:com1}
		X_{im}X_{in}&=-X_{i(n-1)}X_{i(m+1)}, \\ \label{e:com2}
    	X_{im}^*.1&=0, \quad a_{-im}^*.1=0, \\ \label{e:com3}
	    a_{-im}^*X_{in}&=X_{in}a_{-im}^*+X_{i(n-m)},\\ \label{e:com4}
    	X_{im}^*a_{-in}&=a_{-in}X_{im}^*+X_{i(m-n)}^*.
	 \end{align}	
\end{prop}

Recall that for two partitions $\mu \subset\lambda$, if the skew partition $\lambda-\mu$ is connected and contains no $2*2$ boxes, we call $\lambda-\mu$ a {\it rim hook} or {\it border strip}. The height of $\lambda-\mu$ is defined to be one less than the number of rows occupied, denoted by $\mathrm{ht}(\lambda-\mu)$ or $\mathrm{ht}(\lambda/\mu)$. For example, take $\lambda=(5,3,2,2), \mu=(2,1,1)$, then $\lambda-\mu$ can be depicted as
\\
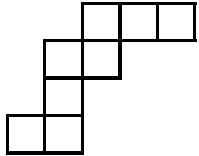
\begin{center}
	\(\young(::~~~,:~~,:~,~~)\)
\end{center}
~\\
which is a $8$-rim hook of $\lambda$ with height 3.
Note that a hook is a special case of the rim hook.

If $\xi$ is a rim hook contained in a Young diagram $\lambda$, then $\lambda\setminus\xi$ denotes the {\it subdiagram} of $\lambda$ obtained by removing $\xi$. It is clear that $\lambda\setminus\xi$ is also a Young diagram, so we also use it refer to the associated partition. In the above example, $\xi=(5, 3, 2, 2)-(2, 1, 1)$ is the rim hook, and the difference partition is
\begin{equation}\label{e:minus}
\lambda\setminus\xi=(2, 1, 1).
\end{equation}

A colored partition $\boldsymbol{\la}=(\la^{(i)})$ is contained in a colored partition $\boldsymbol{\mu}=(\mu^{(i)})$, denoted by
$\boldsymbol{\la}\subset\boldsymbol{\mu}$, if each
constituent partition $\la^{(i)}\subset\mu^{(i)}$.
The difference $\boldsymbol{\xi}=\boldsymbol{\la}-\boldsymbol{\mu}$ is called
a rim-hook if one of the constituents $\la^{(i)}-\mu^{(i)}$
is a rim-hook and all other constituents are identical, i.e., the colored rim-hook $\boldsymbol{\xi}$ is
only supported at the $i$th constituent.

For an integral $l$-tuple $\mu=(\mu_1, \ldots, \mu_l)$, we define the lowering operator $T_n$ ($1\leq n\leq l-1$) by
	 \begin{align*}
	 	T_n: \quad (\mu_1,\cdots \mu_n,\mu_{n+1},\cdots \mu_l) \mapsto (\mu_1,\cdots \mu_{n+1}-1,\mu_n+1,\cdots \mu_l).
	 \end{align*}
Now we give a characterization of $m$-rim hook $\eta$ by using lowering operators. Let $\xi=(\lambda_1,\cdots,\lambda_{i-1},\lambda_{i}-m,\lambda_{i+1},\cdots,\lambda_l)$  be the integral $l$-tuple obtained by the action \eqref{e:com3} of vertex operators.

\begin{lem}\label{l:mstrip}  Any $m$-rim hook $\eta$ can be written as a difference of a partition $\la$ and $\mathrm{ht}(\eta)$ applications of $T_{i}$ to the integral $l$-tuple $(\la_1, \ldots, \la_{i-1}, \la_i-m, \la_{i+1}, \ldots,
\la_l)$.
\end{lem}
\begin{proof} Assume that a $m$-rim hook $\eta$ begins at the $i$-th row of $\lambda$. If its height is 0, i.e., it only occupies the $i$-th row of $\lambda$. Then $m\leq \lambda_i-\lambda_{i+1}$, or $\lambda_i-m\geq \lambda_{i+1}$, so $\xi$ is already a partition and $\eta=\lambda-\xi$. Suppose now $\mathrm{ht}(\eta)=1$ and that
it occupies the $i$-th and $(i+1)$-th rows of $\lambda$. By definition of a rim hook,  $\lambda_i-\lambda_{i+1}+2\leq m\leq \lambda_i-\lambda_{i+2}+1$, so $\lambda_{i+2}-1\leq \lambda_i-m\leq \lambda_{i+1}-2$. Hence
$\eta=\lambda-T_i(\xi)$ and $T_i(\xi)$ is a partition. In general, if the height of an $m$-rim hook is $r\geq 1$, one has
\begin{align}\notag
    \sum_{j=0}^{r-1} &(\lambda_{i+j}-\lambda_{i+j+1}+1)+1 \leq m\\
    &\leq \sum_{j=0}^{r-1}(\lambda_{i+j}-\lambda_{i+j+1}+1)+(\lambda_{i+r}-\lambda_{i+r+1}).
\end{align}
Then $\lambda_{i+r+1}-r\leq \lambda_i-m\leq \lambda_{i+r}-(r+1)$, and we obtain $\eta=\lambda-T_{i+r-1}\cdots T_{i}(\xi)$.
\end{proof}

\begin{lem}\label{l:MN1}
	 For each $j\in I$, any partition $\lambda=(\lambda_1,\lambda_2,\cdots \lambda_l)$, and any $m\in \mathbb{N}$, we have
    \begin{align}\label{e:action}
	a_{-jm}^*s_{j\lambda}
 &=\sum_\mu(-1)^{\mathrm{ht}(\lambda/\mu)}s_{j\mu},
    \end{align}
summed over all partitions $\mu$ such that $\la-\mu$ is an $m$-rim hook.
\end{lem}	
\begin{proof} Repeatedly using \eqref{e:com3} we have
that
	\begin{align}
		a_{-jm}^*s_{j\lambda}=\sum_{i=1}^{l}X_{j\lambda_1}\cdots X_{j\lambda_i-m}\cdots X_{j\lambda_l}.1.
	\end{align}
 The vector $X_{j\lambda_1}\cdots X_{j\lambda_i-m}\cdots X_{j\lambda_l}.1$ is the Schur function associated with integral $l$-tuple $(\lambda_1, \ldots, \lambda_{i-1}, \lambda_i-m, \lambda_{i+1}\ldots, \lambda_l)$. If $\lambda_i-m<\lambda_{i+1}$, we need to straighten out the integral $l$-tuple to partition. We use \eqref{e:com1} to rewrite
 \begin{equation*}
 X_{j\lambda_1}\cdots X_{j\lambda_i-m}\cdots X_{j\lambda_l}.1=-X_{j\lambda_1}\cdots X_{j\lambda_{i+1}-1}X_{j\lambda_i-m+1}\cdots X_{j\lambda_l}.1.
 \end{equation*}
 If $\lambda_i-m+1=\lambda_{i+1}$, then the term vanishes. If
 $\lambda_i-m+1<\lambda_{i+2}$,
 then we repeat the swapping $(\lambda_i-m+1, \lambda_{i+2})\to (\lambda_{i+2}-1, \lambda_{i+2}-m+2)$ until we have a partition. Suppose after $r$
 steps, we reach the partition
 \begin{align*}
	\xi=(\lambda_1,\cdots,\lambda_{i-1},\lambda_{i+1}-1,\cdots,\lambda_{i+r}-1,\lambda_i-m+r,\lambda_{i+r+1},\cdots,\lambda_l).
    \end{align*}
Then the skew partition
$$\lambda-\xi=(0^{i-1}, \la_i-\la_{i+1}+1, \ldots,
\la_{i+r}-\la_i+m-r, \la_{i+r+1}, \ldots, \la_l)$$
is a $m$-rim hook and $s_{j\lambda}=(-1)^rs_{j\xi}$.

Conversely, suppose $\xi$ is a partition such that $\la-\xi$ is a horizontal $m$-strip. By Lemma \ref{l:mstrip} the term $s_{\xi}$ appears
in the expansion of $a_{-jm}^*s_{j\lambda}$.
Therefore \eqref{e:action} is proved.
\end{proof}	

\begin{exmp}
	\begin{align*}
		a_{-1,4}^*s_{1(5,3,2)}&=s_{1(1,3,2)}+s_{1(5,-1,2)}+s_{1(5,3,-2)}\\\nonumber
		&=-s_{1(2,2,2)}-s_{1(5,1,0)}		
	\end{align*}
\end{exmp}

Similarly, a colored partition $\boldsymbol{\mu}\subset\boldsymbol{\lambda}$ such that $\boldsymbol{\lambda}-\boldsymbol{\mu}$ is a {\it colored rim hook} if and only if
one of its constituents is a rim hook and the others are $\phi$. The height of $\boldsymbol{\lambda}-\boldsymbol{\mu}$ is exactly the height of its nonempty constituent, denoted by $\mathrm{ht}(\boldsymbol{\lambda}-\boldsymbol{\mu})$ or
$\mathrm{ht}(\boldsymbol{\lambda}/\boldsymbol{\mu})$. For example, if $\boldsymbol{\lambda}=((2,1),(5,3,2,2),(4))$, and $\boldsymbol{\mu}=((2,1),(2,1,1),(4))$, then $\boldsymbol{\lambda}-\boldsymbol{\mu}$ is an 8-rim hook colored at the second constituent with height $3$.
\\
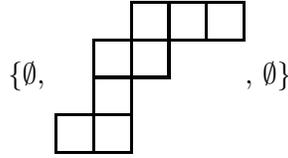
\begin{center}
	$\{\emptyset$, \(\young(::~~~,:~~,:~,~~)\), $\emptyset\}$
\end{center}

If $\boldsymbol{\xi}$ is a colored rim hook contained in a colored Young diagram $\boldsymbol{\lambda}$, then $\boldsymbol{\lambda}\setminus\boldsymbol{\xi}$ denotes the {\it colored subdiagram} of $\boldsymbol{\lambda}$ obtained by removing $\boldsymbol{\xi}$. Similarly we also use
$\boldsymbol{\lambda}\setminus\boldsymbol{\xi}$
to denote the associated colored partition.

Now we are ready to give a vertex algebraic proof of the Murnaghan-Nakayama rule for
the generalized symmetric group $C_k\wr S_n$. The formula \eqref{e:MN-rule1} was first proved by group theoretical method in \cite{AK}.
\begin{thm}\label{t:MN-gensym}
	Given colored partitions $\boldsymbol{\lambda}=(\lambda^{(0)},\lambda^{(1)},\cdots,\lambda^{(k-1)})$ and $\boldsymbol{\rho}=(\rho^{(0)},\rho^{(1)},\cdots,\rho^{(k-1)})$ of $n$,
 where $\rho^{(s)}=(\rho^{(s)}_1,\cdots,\rho^{(s)}_m,\cdots,\rho^{(s)}_{l(s)})$. Then the value of the irreducible character $\chi^{\boldsymbol{\lambda}}$ of $C_k\wr S_n$ at the class $\boldsymbol{\rho}$ is given by
	\begin{align}\label{e:MN-rule1}
		\chi^{\boldsymbol{\lambda}}_{\boldsymbol{\rho}}=\sum_{j=0}^{k-1}\sum_{\boldsymbol{\xi_j}}(-1)^{\mathrm{ht}(\boldsymbol{\xi_j})}\omega^{-sj}\chi^{\boldsymbol{\lambda}\backslash\boldsymbol{\xi_j}}_{\boldsymbol{\rho}\setminus\rho^{(s)}_m},
	\end{align}
where ${\boldsymbol{\xi_j}}$ runs through all colored $\rho^{(s)}_m$-rim hooks contained in ${\boldsymbol{\lambda}}$ that are supported at the $j$-th constituent.
\end{thm}	
\begin{proof} Fixing $j\in I$, it follows from \eqref{e:2.6} that:
 $$a_{-j\rho^{(s)}_m}^*s_{i\lambda^{(i)}}=0,$$
 if $i\neq j$. Then nonzero contributions only come from
 $$a_{-j\rho^{(s)}_m}^*s_{\boldsymbol{\lambda}}=s_{0\lambda^{(0)}}\cdots (a_{-j\rho^{(s)}_m}^*s_{j\lambda^{(j)}})\cdots s_{(k-1)\lambda^{(k-1)}}.$$
  By Lemma \ref{l:MN1}, $a_{-j\rho^{(s)}_m}^*s_{\boldsymbol{\lambda}}$ is expanded into an alternating sum of $s_{\boldsymbol{\lambda}\backslash\boldsymbol{\xi_j}}$
  such that $\boldsymbol{\xi_j}$ is a colored $\rho^{(s)}_m$-rim hook supported at the $j$-th constituent. By Theorem \ref{t:FJW} and converting back to the basis $ {\widetilde{a}_{-s\rho^{(s)}_m}}$, using
     \begin{align*}
         {\widetilde{a}_{-s\rho^{(s)}_m}}=\sum_{j=0}^{k-1}\omega^{-sj}a_{-j\rho^{(s)}_m},
     \end{align*}
    we have shown the theorem.
\end{proof}

\begin{exmp} In the group $C_3\wr S_{12}$, we have
	\begin{align*}
		\chi^{((4,1),(3,1,1),(2))}_{((5,2),(3),(1,1))}=-\chi^{(\emptyset,(3,1,1),(2))}_{((2),(3),(1,1))}+\chi^{((4,1),\emptyset,(2))}_{((2),(3),(1,1))}.
	\end{align*}
 Similarly in $C_3\wr S_{14}$, one has
	\begin{align*}
		\chi^{((3,2),(4,2,1),(2))}_{((2,1),(4,3),(3,1))}=-\chi^{((1),(4,2,1),(2))}_{((2,1),(3),(3,1))}-\omega^2\chi^{((3,2),(1,1,1),(2))}_{((2,1),(3),(3,1))},
	\end{align*}
where $\omega$ is the 3-rd primitive root of unity.
\end{exmp}

Similar to the symmetric group, we have the following result by applying Theorem \ref{t:MN-gensym}. Recall that a {\it colored hook} is a special colored partition if and only if one of its constituent is a hook and the others are $\emptyset$. The length of arm and leg of a colored hook are exactly those of its nonempty constituent respectively. Obviously a colored hook is a special colored rim hook.

\begin{prop}
    Let $\boldsymbol{\kappa}_i$ be a colored partition supported only at the $i$-th constituent $(n)$ and $\omega$ the $k$-th primitive root of unity. Then the irreducible character $\chi^{\boldsymbol{\lambda}}$ of $C_k\wr S_n$ satisfies
    \begin{align}
        \chi^{\boldsymbol{\lambda}}(\boldsymbol{\kappa}_i)=(-1)^r\omega^{-ij} \quad or \quad 0
    \end{align}
    according as $\boldsymbol{\lambda}$ is a colored hook or not, where $j$ is the index of its non-empty constituent and $r$ is its length of leg.
\end{prop}

Now we study how to decompose the irreducible characters using the dual vertex operator in the other direction.

For each partition $\rho=(\rho_1,\rho_2,\cdots,\rho_l)$, a {\it subpartition} denoted by $\mu\vartriangleleft\rho$ is $\mu=\emptyset$ or $\mu=(\rho_{i_1},\rho_{i_2},\cdots,\rho_{i_s})$, for some $1\leqslant i_1<\cdots<i_s\leqslant l$. Obviously, there are $2^{l(\rho)}$ subpartitions $\mu$ such that $\mu\vartriangleleft\rho$ and
$\rho\backslash\mu$ is also a partition. For example, if $\rho=(4,2,2,1)$, and $\mu=(2,1)$, then $\mu\vartriangleleft\rho$ and $\rho\backslash\mu=(4,2)$. Moreover, $\mu\cup\tau$ of partitions $\mu$ and $\tau$, is defined as the union of parts of $\mu$ and $\tau$ in the descending order. 

\begin{lem}\label{l:3.8}
For any partition $\rho=(\rho_1,\rho_2,\cdots,\rho_l)$, $m\in \mathbb{N}$, and
     $i,j\in I$, one has that
     \begin{align*}
     	X_{jm}^*\widetilde{a}_{-i\rho}=\sum_{\mu\vartriangleleft\rho}\omega^{-ij\cdot(l(\mu))}\widetilde{a}_{-i\rho\backslash\mu}X_{j(m-|\mu|)}^*,
     \end{align*}
summed over all subpartitions of $\rho$.
\end{lem}
\begin{proof}
    For each $1\leq s\leq l$, it follows from the linearity of $\widetilde{a}_{-i\rho_s}$ and the action of \eqref{e:com4} that
	\begin{align}
		X_{jm}^*\widetilde{a}_{-i\rho_s}=\widetilde{a}_{-i\rho_s}X_{jm}^*+\omega^{-ij}X_{j(m-\rho_s)}^*.
	\end{align}
 
    In the case of length $2$ partition, the expression becomes
    \begin{align*}
        \widetilde{a}_{-i\rho}X_{jm}^*+\omega^{-ij}\widetilde{a}_{-i\rho_1}X_{j(m-\rho_2)}^*+\omega^{-ij}\widetilde{a}_{-i\rho_2}X_{j(m-\rho_1)}^*+\omega^{-ij\cdot 2}X_{j(m-|\rho|)}^*.
    \end{align*}
    Using induction on $l(\rho)$, we get the lemma.
\end{proof}

For two colored partitions $\boldsymbol{\mu}$ and $\boldsymbol{\rho}$, if for each $i\in I$, such that $\mu^{(i)}\vartriangleleft\rho^{(i)}$, we write $\boldsymbol{\mu}\vartriangleleft\boldsymbol{\rho}$. The set minus $\boldsymbol{\rho}\backslash\boldsymbol{\mu}$ (removing the empty rows) is also a colored partition. For example, if $\boldsymbol{\rho}=((5,3,1),(4,2,2,1),(4))$, and $\boldsymbol{\mu}=((3),(2,1),(4))$, then $\boldsymbol{\mu}\vartriangleleft\boldsymbol{\rho}$ and $\boldsymbol{\rho}\backslash\boldsymbol{\mu}=((5,1),(4,2),\emptyset)$.

A colored partition $\boldsymbol{\rho}$ naturally gives rise to a partition by forgetting the colors, we call this partition the {\it rearrangement} of $\boldsymbol{\rho}$. For example, take $\boldsymbol{\rho}=((5,1),\emptyset,(4,2))$, it can be rearranged into a regular partition $(5,4,2,1)$. 
We also define the length of a colored partition $l(\boldsymbol{\rho})$ as the sum of $l(\rho^{(i)})$.
Recall that the weighted length 
$\eta(\boldsymbol{\rho})$ is defined as $\sum_iil(\rho^{(i)})$.
 The juxtaposition $\boldsymbol{\mu}\cup \boldsymbol{\tau}$ of two colored partitions is defined as $(\mu^{(i)}\cup \tau^{(i)})$.

 Now we are ready to give our second main result--the dual Maunaghan-Nakayama rule for the generalized symmetric group.

\begin{thm}\label{t:3.9}
	Given colored partitions $\boldsymbol{\lambda}=(\lambda^{(0)},\lambda^{(1)},\cdots,\lambda^{(k-1)})$ and $\boldsymbol{\rho}=(\rho^{(0)},\rho^{(1)},\cdots,\rho^{(k-1)})$ of $n$. Then for any fixed $j\in I$,
 the irreducible character value of $C_k\wr S_n$ is given by
    \begin{align}\label{e:MN-rule2}
    	\chi^{\boldsymbol{\lambda}}_{\boldsymbol{\rho}}=\sum_{\substack{\boldsymbol{\mu}\vartriangleleft\boldsymbol{\rho},\|\boldsymbol{\mu}\|\geq\lambda^{(j)}_1\\\boldsymbol{\tau}\vdash \|\boldsymbol{\mu}\|-\lambda^{(j)}_1}}\omega^{(\eta(\boldsymbol{\tau})-\eta(\boldsymbol{\mu}))j}\frac{{(-1)}^{l(\boldsymbol{\tau})}}{k^{l(\boldsymbol{\tau})}z_{\tau}}\chi^{\boldsymbol{\lambda}\setminus \lambda^{(j)}_1}_{(\boldsymbol{\rho}\backslash\boldsymbol{\mu})\cup\boldsymbol{\tau}}
    \end{align}
where the sum is over colored partitions $\boldsymbol{\mu}$ and $\boldsymbol{\tau}$ such that $\boldsymbol{\mu}\vartriangleleft\boldsymbol{\rho}$, with weight bigger than $\lambda^{(j)}_1$, and $\boldsymbol{\tau}$ of weight $\|\boldsymbol{\mu}\|-\lambda^{(j)}_1$. $z_{\tau}$ is defined in \eqref{e:innerprod}, $\tau$ is the rearrangement of $\boldsymbol{\tau}$, and $\eta(\boldsymbol{\tau})$, $\eta(\boldsymbol{\mu})$ are the weighted lengths $\sum_ii l(\tau^{(i)})$, $\sum_ii l(\mu^{(i)})$ respectively.
\end{thm}
\begin{proof}
	Applying Lemma \ref{l:3.8} repeatedly, one has
	\begin{align}\label{p:1}
		X_{jm}^*\widetilde{a}_{-\boldsymbol{\rho}}=\sum_{\boldsymbol{\mu}\vartriangleleft\boldsymbol{\rho}}\omega^{-\eta(\boldsymbol{\mu})j}\widetilde{a}_{-\boldsymbol{\rho}\backslash\boldsymbol{\mu}}X_{j(m-\|\boldsymbol{\mu}\|)}^*.1,
	\end{align}
summed over all colored partitions $\boldsymbol{\mu}$ such that $\boldsymbol{\mu}\vartriangleleft\boldsymbol{\rho}$.
    By \eqref{e:com2}, nonzero items survive only when $m-\|\boldsymbol{\mu}\| \leq 0$. Taking the coefficient of $z^{n}$ in $X_j(z)^*$, we have that
    \begin{align}\label{p:2}
    	X_{jn}^*.1=\sum_{\rho\vdash n}\frac{{(-1)}^{l(\rho)}}{z_\rho}a_{-j\rho} \quad for \quad n\geq 0.
    \end{align}
    By \eqref{e:2.9} and \eqref{e:2.11}, for a partition $\tau=(\tau_1, \tau_2, \cdots, \tau_{l(\tau)})$ we have
    \begin{align}
        \label{e:3.20}
		a_{-j\tau}&=\frac{1}{k^{l(\tau)}}\prod_{s=1}^{l(\tau)}\left(\sum_{i_s=0}^{k-1}\omega^{ji_s}\widetilde{a}_{-i_s\tau_s}\right)
	\\ \nonumber
        &=\frac{1}{k^{l(\tau)}}\sum_{(i_1,\cdots i_{l(\tau)})\in I^{l(\tau)}}\omega^{|i|j}\widetilde{a}_{-i_1\tau_1}\cdots \widetilde{a}_{-i_{l(\tau)}\tau_{l(\tau)}},
        \end{align}
         where $(i_1,\cdots i_{l(\tau)})$ runs through the Cartesian product $I^{l(\tau)}$ and
         $|i|=i_1+\cdots +i_{l(\tau)}$. 
    We rewrite the $k^{l(\tau)}$ summands by reordering factors according to their colors, where the pairs
    $(i_1, \tau_1), (i_2, \tau_2), \ldots, (i_{l(\tau)}, \tau_{l(\tau)})$
    are rearranged according to their colors. Then
    the composition $(\tau_1, \tau_2, \ldots, \tau_{l(\tau)})$
    becomes a colored partition $\boldsymbol{\tau}$ and the summation of its
    colors $|i|=i_1+\cdots+i_{l(\tau)}$ becomes $\eta(\boldsymbol{\tau})$. Therefore
    \begin{equation}\label{e:conversion}
    a_{-j\tau}=\frac1{k^{l(\tau)}}\sum_{\boldsymbol{\tau}}
    \omega^{j\eta(\boldsymbol{\tau})}\tilde{a}_{-\boldsymbol{\tau}},
    \end{equation}
    where the sum runs through all colored partitions arising from $\tau$. Plugging \eqref{e:conversion} into \eqref{e:3.20} and using \eqref{p:1} and \eqref{p:2}, we have derived the recurrence formula. 
    \end{proof}

    We remark that the dual Murnaghan-Nakayama rule of $C_k\wr S_n$ offers a different iterative route from the
    Ariki-Koike version by directly simplifying the indexing
    colored-partition of the irreducible representation by removing
    the largest row in a fixed constituent, while the usual Murnaghan-Nakayama rule
    is iterated by removing all possible rim-hooks from the indexing partition in all constituents. 
    
    When $k=1$, we have the following dual Maunaghan-Nakayama rule for the symmetric group, which seems new to the authors' knowledge.

\begin{thm}\label{c:3.10}
	The irreducible character value of $S_n$ is given by
	\begin{align}
		\chi^{\lambda}_{\rho}=\sum_{\substack{\mu\vartriangleleft\rho,|\mu|\geq\lambda_1\\\tau\vdash |\mu|-\lambda_1}}\frac{{(-1)}^{l(\tau)}}{z_\tau}\chi^{\lambda\setminus \lambda_1}_{(\rho\backslash\mu)\cup\tau},
	\end{align}
 summed over partitions $\mu$ and $\tau$ such that $\mu$ is a subpartition  of $\rho$ with weight bigger than $\lambda_1$,
 and $\tau$ of weight $|\mu|-\lambda_1$.
\end{thm}

\begin{exmp} Using the recurrence relation, we see that in $C_2\wr S_4$
	\begin{align*}
		\chi^{((1),(3))}_{((2,1),(1))}=-\frac{1}{2}\chi^{((1),\emptyset)}_{((1),\emptyset)}+\frac{1}{2}\chi^{((1),\emptyset)}_{(\emptyset,(1))}.
	\end{align*}
 Similarly in $C_3\wr S_8$, we have
	\begin{align}\notag
		\chi^{((2),(5),(1))}_{((3,1),(2),(2))}=&\frac{8}{9}\chi^{((2),\emptyset,(1))}_{((3),\emptyset,\emptyset)}-\frac{1}{9}\omega\chi^{((2),\emptyset,(1))}_{(\emptyset,(3),\emptyset)}-\frac{1}{9}\omega^2\chi^{((2),\emptyset,(1))}_{(\emptyset,\emptyset,(3))}\\\nonumber
		&-\frac{1}{9}\chi^{((2),\emptyset,(1))}_{((2,1),\emptyset,\emptyset)}-\frac{5}{18}\omega^2\chi^{((2),\emptyset,(1))}_{(\emptyset,(2,1),\emptyset)}-\frac{5}{18}\omega\chi^{((2),\emptyset,(1))}_{(\emptyset,\emptyset,(2,1))}\\\nonumber
		&+\frac{1}{18}\omega\chi^{((2),\emptyset,(1))}_{((2),(1),\emptyset)}+\frac{1}{18}\omega^2\chi^{((2),\emptyset,(1))}_{((2),\emptyset,(1))}-\frac{5}{18}\chi^{((2),\emptyset,(1))}_{(\emptyset,(2),(1))}\\\nonumber
		&+\frac{5}{9}\omega\chi^{((2),\emptyset,(1))}_{((1),(2),\emptyset)}+\frac{5}{9}\omega^2\chi^{((2),\emptyset,(1))}_{((1),\emptyset,(2))}-\frac{5}{18}\chi^{((2),\emptyset,(1))}_{(\emptyset,(1),(2))}\\\nonumber
		&+\frac{4}{81}\chi^{((2),\emptyset,(1))}_{((1,1,1),\emptyset,\emptyset)}-\frac{1}{162}\chi^{((2),\emptyset,(1))}_{(\emptyset,(1,1,1),\emptyset)}-\frac{1}{162}\chi^{((2),\emptyset,(1))}_{(\emptyset,\emptyset,(1,1,1))}\\\nonumber
        &+\frac{5}{54}\omega\chi^{((2),\emptyset,(1))}_{((1,1),(1),\emptyset)}+\frac{5}{54}\omega^2\chi^{((2),\emptyset,(1))}_{((1,1),\emptyset,(1))}-\frac{1}{54}\omega\chi^{((2),\emptyset,(1))}_{(\emptyset,(1,1),(1))}\\\nonumber
        &-\frac{1}{54}\omega^2\chi^{((2),\emptyset,(1))}_{(\emptyset,(1),(1,1))}+\frac{1}{27}\omega^2\chi^{((2),\emptyset,(1))}_{((1),(1,1),\emptyset)}+\frac{1}{27}\omega\chi^{((2),\emptyset,(1))}_{((1),\emptyset,(1,1))}\\\nonumber
        &+\frac{2}{27}\chi^{((2),\emptyset,(1))}_{((1),(1),(1))},
	\end{align}
where $\omega$ is the 3-rd primitive root of unity.
For example, in the expansion of $\chi^{((2),\emptyset,(1))}_{(\emptyset,(2,1),\emptyset)}$, there are only two terms: $\boldsymbol{\mu}=((3,1),\emptyset,(2)), \boldsymbol{\tau}=(\emptyset,(1),\emptyset)$ and $\boldsymbol{\mu'}=((3,1),(2),(2)), \boldsymbol{\tau'}=(\emptyset,(2,1),\emptyset)$. Therefore its coefficient is a sum of two terms.

\end{exmp}

The following is a special case of a direct consequence of Theorem \ref{t:3.9}.
\begin{prop} Let $\boldsymbol{\kappa}_i$ be a colored partition supported only at the $i$th constituent $(n)$.
Then for any colored partition $\boldsymbol{\rho}=(\rho^{(0)},\rho^{(1)},\cdots,\rho^{(k-1)})$ of $n$, we have that
    \begin{align}
        \chi^{\boldsymbol{\kappa}_i}_{\boldsymbol{\rho}}=\omega^{-\eta(\boldsymbol{\rho})i},
    \end{align}
    where $\eta(\boldsymbol{\rho})$ is the weight length
    $\sum_i il({\rho}^{(i)})$.
\end{prop}

\section{Further examples}

For completeness, we first calculate the degree of any irreducible character of $C_k\wr S_n$ \cite{IJS}. Recall that the inner product in $V$ is given by:
\begin{align}
	\langle a_{-\boldsymbol{\lambda}},a_{-\boldsymbol{\mu}} \rangle=\prod_{i=0}^{k-1}\langle a_{-\lambda^{(i)}},a_{-\mu^{(i)}}\rangle=\delta_{\boldsymbol{\lambda}\boldsymbol{\mu}}\prod_{i=0}^{k-1}z_{\lambda^{(i)}}.
\end{align}
Under this inner product the wreath Schur functions are orthonormal:
\begin{align}
	\langle s_{\boldsymbol{\lambda}},s_{\boldsymbol{\mu}} \rangle=\prod_{i=0}^{k-1}\langle s_{\lambda^{(i)}},s_{\mu^{(i)}}\rangle=\delta_{\boldsymbol{\lambda}\boldsymbol{\mu}}.
\end{align}
Then we have
\begin{align}\label{e:inner}
	\langle s_{\boldsymbol{\lambda}},a_{-\boldsymbol{\mu}} \rangle=\prod_{i=0}^{k-1}\langle s_{\lambda^{(i)}},a_{-\mu^{(i)}}\rangle.
\end{align}

The following result is well-known for the wreath products and is included here using the vertex operator approach.
\begin{thm} \cite{O} For any
	colored partition $\boldsymbol{\lambda}=(\lambda^{(0)},\lambda^{(1)},\cdots,\lambda^{(k-1)})$ of $n$, the degree of the irreducible character $\chi^{\boldsymbol{\lambda}}$ of $C_k\wr S_n$ is given by
	\begin{align}
		\chi^{\boldsymbol{\lambda}}(\boldsymbol{1})=\frac{n!}{\prod_{i=0}^{k-1}h(\lambda^{(i)})},
	\end{align}
    where $\boldsymbol{1}=((1^n),\emptyset,\cdots,\emptyset)$ and $h(\lambda^{(i)})$ is the hook length of $\lambda^{(i)}$.
\end{thm}	
\begin{proof}
	By definition,
		$\chi^{\boldsymbol{\lambda}}(\boldsymbol{1})=\langle s_{\boldsymbol{\lambda}}, (\widetilde{a}_{-01})^n \rangle=\langle s_{\boldsymbol{\lambda}}, (\sum_{j=0}^{k-1}a_{-j1})^n \rangle$.
    The non-zero items are only
	\begin{align*}
		\langle s_{\boldsymbol{\lambda}},\prod_{j=0}^{k-1}a_{-j1}^{|\lambda^{(j)}|} \rangle=\prod_{j=0}^{k-1}\langle s_{\lambda^{(j)}},a_{-j1}^{|\lambda^{(j)}|} \rangle=\prod_{j=0}^{k-1}\frac{|\lambda^{(j)}|!}{h(\lambda^{(j)})},
	\end{align*}
    where the second equality has used the degree formula of irreducible characters in $S_{\lambda^{(j)}}$. Since the coefficient of $\prod_{j=0}^{k-1}a_{-j1}^{|\lambda^{(j)}|}$ in $(\sum_{j=0}^{k-1}a_{-j1})^n$ equals to
    \begin{align*}
    	C_n^{|\lambda^{(0)}|}C_{n-|\lambda^{(0)}|}^{|\lambda^{(1)}|}\cdots=\frac{n!}{\prod_{j=0}^{k-1}|\lambda^{(j)}|!},
    \end{align*}
  and the theorem is proved.
\end{proof}	

\begin{cor}
    Given colored partitions $\boldsymbol{\lambda}=(\lambda^{(0)},\lambda^{(1)},\cdots,\lambda^{(k-1)})$ of $n$, let $\boldsymbol{\kappa}^j$ be a colored partition supported only at the $j$-th constituent $(1^n)$, and $\omega$ be the $k$-th primitive root of unity. Then we have
    \begin{align}
        \chi^{\boldsymbol{\lambda}}_{\boldsymbol{\kappa}^j}=\frac{\omega^{-\mathrm{deg}(\boldsymbol{\lambda})j}n!}{\prod_{i=0}^{k-1}h(\lambda^{(i)})},
    \end{align}
    where $\mathrm{deg}(\boldsymbol{\lambda})$ is defined as the sum of $i |\lambda^{(i)}|$.
\end{cor}

In the dual case, one has the following result.
\begin{prop}
    Given a colored partition $\boldsymbol{\rho}=(\rho^{(0)},\rho^{(1)},\cdots,\rho^{(k-1)})$ of $n$, the irreducible character value is
    \begin{align}\label{e:char-}
		\chi^{\boldsymbol{\kappa}^j}_{\boldsymbol{\rho}}=\omega^{-\eta(\boldsymbol{\rho})j}(-1)^{l(\boldsymbol{\rho})+n},
	\end{align}
    where $\boldsymbol{\kappa}^j$ is a colored partition supported only at the $j$-th constituent $(1^n)$, $\eta(\boldsymbol{\rho})$ is the weighted length $\sum_i il(\rho^{(i)})$.
\end{prop}
\begin{proof} In view of \eqref{e:inner}, the only nonzero factor of $\langle s_{\boldsymbol{\kappa}^j}, \widetilde{a}_{\boldsymbol{\rho}} \rangle$ is
    \begin{align*}
        \langle s_{j(1^n)}, \prod_{i=0}^{k-1}\omega^{-ij|\rho^{(i)}|}a_{-j\rho^{(i)}} \rangle
    \end{align*}
    since $|\lambda^{(j)}|$ exhausts $n$.
    Note that $\prod_{i=0}^{k-1}a_{-j\rho^{(i)}}$ can be written as $a_{-j\rho}$, where $\rho$ is the rearrangement of $\boldsymbol{\rho}$.

    For the character value $\chi^{(1^n)}_{\rho}$, we apply Theorem \ref{c:3.10} with induction on $n$. First $\chi^{(1)}_{(1)}=1$ is clear. The inductive hypothesis says that
    $$\chi^{(1^{n-1})}_{(\rho\backslash\mu)\cup\tau}=(-1)^{l(\rho)-l(\mu)+l(\tau)+n-1}.$$
    Note that
    $$\sum_{\lambda\vdash n}\frac{1}{z_\lambda}=1.$$
    Then
    $$\chi^{(1^n)}_{\rho}=(-1)^{l(\rho)+n-1}\sum_{\mu\vartriangleleft\rho,|\mu|\geq 1}(-1)^{l(\mu)},$$
    where the sum is over all nonempty subpartitions $\mu\vartriangleleft\rho$, which is an alternating sum of binomial coefficients:
    $$\sum_{k=1}^{l(\rho)}(-1)^k\binom{l(\rho)}{k}=-C_{l(\rho)}^0=-1,$$
    Therefore
    $\chi^{(1^n)}_{\rho}=(-1)^{l(\rho)+n}$. Obviously $l(\boldsymbol{\rho})=l(\rho)$, and $\prod_{i=0}^{k-1}\omega^{-ij|\rho^{(i)}|}$ equals to $\omega^{-\eta(\boldsymbol{\rho})j}$, subsequently we have  shown
    $\chi^{\boldsymbol{\kappa}^j}_{\boldsymbol{\rho}}$
    is given as stated.
\end{proof}

Next we move on to study a relationship between irreducible character values of $C_k\wr S_n$ and those of $S_{kn}$ mod $k$.  Note that $k$ does not have to be a prime.

Recall that the {\it $k$-core $\gamma_\la$} of partition $\la=(\la_1, \la_2,\cdots, \la_l)$ is defined by removing all $k$-rim hooks of $\la$. For example, let $\la=(4,2,1)$, $k=3$,
\\
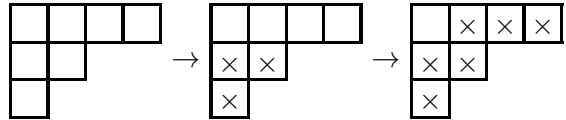
\begin{center}
	\(\young(~~~~,~~,~)\) $\to$ \(\young(~~~~,\times\times,\times)\) $\to$ \(\young(~\times\times\times,\times\times,\times)\)
\end{center}
~\\
then the 3-core of $\la$ is $(1)$. 

Let $T=T(\lambda)$ be the set of paths from partition $\la$ to its core $\gamma_{\la}$, where each path is a procedure of removing $k$-rim hooks. For each path $t\in T$, denote by $\mathrm{ht}(t)$ the sum of heights of $k$-rim hooks in path $t$. It is known that $\sigma_{\la}: =(-1)^{\mathrm{ht}(t)}$ is well defined and independent of the paths in $T(\lambda)$ \cite[Prop. 2.2]{MO}. For example, the 3-core of $\lambda=(5,4,3,2)$ is (2). Consider two paths $t_1, t_2\in T$ as follows. $t_1$ has $\mathrm{ht}(t_1)=5$:
\\
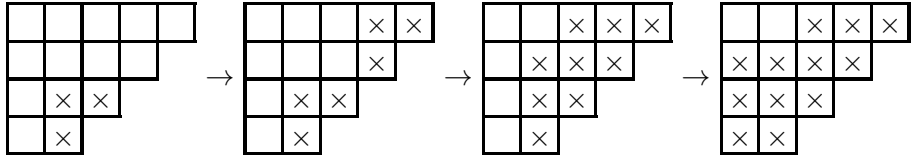
	\begin{center}
		\(\young(~~~~~,~~~~,~\times\times,~\times)\) $\to$ \(\young(~~~\times\times,~~~\times,~\times\times,~\times)\) $\to$ \(\young(~~\times\times\times,~\times\times\times,~\times\times,~\times)\) $\to$
		\(\young(~~\times\times\times,\times\times\times\times,\times\times\times,\times\times)\)
	\end{center}
while $t_2$ has $\mathrm{ht}(t_2)=3$:
\\
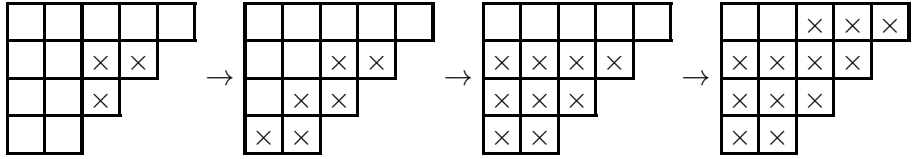
 \begin{center}
		\(\young(~~~~~,~~\times\times,~~\times,~~)\) $\to$ \(\young(~~~~~,~~\times\times,~\times\times,\times\times)\) $\to$ \(\young(~~~~~,\times\times\times\times,\times\times\times,\times\times)\) $\to$
		\(\young(~~\times\times\times,\times\times\times\times,\times\times\times,\times\times)\)
	\end{center}
so $\sigma_{\lambda}$ equals to $-1$.

The {\it $k$-quotient $\boldsymbol{\beta}_\la$} of $\la$ can be constructed as follows. Choose $m$ as a multiple of $k$
such that $m\geq l(\lambda)$, let $\delta_m=(m-1, m-2,\cdots,1,0)$, $\xi=\la+\delta_m$, which is a strict partition. For each $0\leq r\leq k-1$, suppose there are $m_r$ parts $\xi_i$ that are congruent to $r$ mod $k$. Each of these $\xi_i \,(1\leq i\leq m_r$) is divided by $k$
and reordered as
$$\xi_i=k\cdot \xi^{(r)}_s+r$$
such that $\xi^{(r)}_1>\xi^{(r)}_2>\cdots>\xi^{(r)}_{m_r}\geq 0$. So for each $0\leq r\leq k-1$, we obtain a partition
$\la^{(r)}=(\la^{(r)}_1, \la^{(r)}_2,\cdots,\la^{(r)}_{m_r})$ by 
letting $\la^{(r)}_s=\xi^{(r)}_s-m_r+s$ ($1\leq s\leq m_r$).
Now the colored partition $(\la^{(0)}, \la^{(1)},\ldots, \la^{(k-1)})=\boldsymbol{\beta}_\la$ is 
the $k$-quotient of $\la$.

For example, given $\la=(4,2,1)$ and $k=3$ as above. Take $m=3$, then $\xi=(6,3,1)$, $m_0=2, m_1=1$ and $m_2=0$. We have
$$\la^{(0)}_1=2-2+1=1, \quad \la^{(0)}_2=1-2+2=1; \quad \la^{(1)}_1=0-1+1=0,$$
thus the 3-quotient of $\la$ is $((1,1), \emptyset, \emptyset)$.

The above procedure shows that any partition $\la$ determines its $k$-core $\gamma_\la$ and $k$-quotient $\boldsymbol{\beta}_\la$ uniquely. Moreover,
$$|\la|=k\cdot \Arrowvert\boldsymbol{\beta}_\la\Arrowvert+|\gamma_\la|.$$
In fact, $\la$ also can be determined uniquely by $\gamma_\la$ and $\beta_\la$ conversely (cf. \cite[I.I. Ex. 8]{Mac}).

For convenience, we define the {\it standard form} of an irreducible character value of
$C_k\wr S_n$. According to Theorem \ref{t:MN-gensym}, each $\chi_{\boldsymbol{\rho}}^{\boldsymbol{\la}}$ can be written as
$$c_0+c_1\omega+\cdots+c_{k-1}\omega^{k-1}, c_i\in \mathbb{Z}$$
since the coefficients of the recursive formula lie in $\mathbb{Z}[\omega]$. However, $k$-tuple $(c_0, c_1, \cdots, c_{k-1})$ is not unique due to $1+\omega+\cdots+\omega^{k-1}=0$. Indeed,
$$c_0+c_1\omega+\cdots+c_{k-1}\omega^{k-1}=c_0'+c_1'\omega+\cdots+c_{k-1}'\omega^{k-1}$$
if and only if $c_0-c_0'=c_1-c_1'=\cdots=c_{k-1}-c_{k-1}'$. For example, the values of $\chi_{(\emptyset, (1), (1,1))}^{(\emptyset, (2), (1))}$ can be
\begin{gather*}
\vdots\\
2+\omega+3\omega^2\\
1+2\omega^2\\
-\omega+\omega^2\\
\vdots
\end{gather*}
The standard form is the {\it unique} form such that the sum of $c_i$ lies in the color set $I=\{0, 1, \ldots, k-1\}$. We also denote  by $d(\chi_{\boldsymbol{\rho}}^{\boldsymbol{\la}})$ the sum of $c_i$ in the standard form. Then $-\omega+\omega^2$ is the standard form in the above example, and $d(\chi_{(\emptyset, (1), (1,1))}^{(\emptyset, (2), (1))})=0$.

\begin{thm} (Relation between $C_k\wr{S}_n$ and $S_{kn}$)
Given colored partitions $\boldsymbol{\la}$ and $\boldsymbol{\rho}$ of $n$. Let $\la$ be the partition of $kn$ with the $k$-core $\emptyset$ and $k$-quotient $\boldsymbol{\la}$, and let $\rho$ be the rearrangement of $k\boldsymbol{\rho}$.
Then
 $$d(\chi_{\boldsymbol{\rho}}^{\boldsymbol{\la}})=\sigma_{\la}(\chi_{\rho}^{\la} \quad mod \quad k).$$
\end{thm}

\begin{proof}
Let $P$ be the set of paths from $\boldsymbol{\la}$ to $\emptyset$ that remove colored $|\rho^{(0)}_1|$-rim hooks, colored $|\rho^{(0)}_2|$-rim hooks,..., colored $|\rho^{(k-1)}_{l(k-1)-1}|$-rim hooks, colored $|\rho^{(k-1)}_{l(k-1)}|$-rim hooks.
By Theorem \ref{t:MN-gensym} it follows that the irreducible character value $\chi_{\boldsymbol{\rho}}^{\boldsymbol{\la}}$
of $C_k\wr S_n$ can be written as
$$\chi_{\boldsymbol{\rho}}^{\boldsymbol{\la}}=\sum_{p\in P}\sigma_p\omega(p),$$
where the sum is over all paths $p\in P$ and $\sigma_p: =(-1)^{\mathrm{ht}(p)}$ with 
$\mathrm{ht}(p)$ being the sum of all colored rim-hook heights in path $p$. Note that $\omega(p)$ equals to some power of $\omega$ depending on $p$. Suppose the standard form of $\chi_{\boldsymbol{\rho}}^{\boldsymbol{\la}}$ is of the following form: $\sum_{i=0}^{k-1}c_i\omega^i$.

On the other hand, let $\tilde{P}$ be the set of paths from $\la$ to $\emptyset$ that remove $k|\rho_1|$-rim hooks, $k|\rho_2|$-rim hooks,..., and $k|\rho_{l(\la)}|$-rim hooks. By the ordinary Murnaghan-Nakayama rule of $S_{kn}$, we have that
$$\chi_{\rho}^{\la}=\sum_{\tilde{p}\in \tilde{P}}\sigma_{\tilde{p}},$$
where $\sigma_{\tilde{p}}=(-1)^{\mathrm{ht}(\tilde{p})}$ is defined similarly as $\sigma_p$. It is known that
$\sigma_p$ and $\sigma_{\tilde{p}}$ are related by
the fundamental relation \cite{O}:
\begin{align}
\sigma_p=\sigma_{\la}\sigma_{\tilde{p}}.
\end{align}
and there is a bijection between the sets $P$ and $\tilde{P}$. This implies that 
$\sum_ic_i$ in the standard form equals to 
$\sigma_{\la}\cdot \chi_{\rho}^{\la}$ mod $k$, which finishes the proof.
\end{proof}

Finally we discuss how to implement our iterative formulas using the SageMath program.

An algorithm design for Theorem \ref{t:MN-gensym} to display all coefficients of the MN rule is as follows:\\
Step 1) Given two colored partitions of length $k$ and weight $n$, named $\boldsymbol{\lambda}$ and $\boldsymbol{\rho}$;\\
Step 2) Assume the largest part (of weight $m$) in $\boldsymbol{\rho}$ be removed (to reduce the times of iterations), list all colored partitions of weight $n-m$ included in $\boldsymbol{\lambda}$;\\
Step 3) Filter out colored partitions $\boldsymbol{\lambda_j}$ such that skew partition $\boldsymbol{\lambda}-\boldsymbol{\lambda_j}$ is a rim hook, and compute the corresponding coefficient.

An algorithm outline of listing all coefficients for 
Theorem \ref{t:3.9} goes as follows.\\
Step 1) Given two colored partitions of length $k$ and weight $n$, named $\boldsymbol{\lambda}$ and $\boldsymbol{\rho}$;\\
Step 2) Identify and remove the largest part  in $\boldsymbol{\lambda}$ (of weight $s$), filter out sub-partitions $\boldsymbol{\mu}$ of $\boldsymbol{\rho}$ with weight less than or equal to $n-s$, and corresponding partitions $\tau$ of size $n-s-|\boldsymbol{\mu}|$, then compute the corresponding coefficients;\\
Step 3) List all possible colored partitions $\boldsymbol{\tau}$ corresponding $\tau$, also compute each of coefficient.
\medskip

To end this section, we list the character tables of $C_3\wr S_1$, $C_3\wr S_2$, $C_3\wr S_3$, we denote by $\omega$ the 3-rd primitive root of unity.

\begin{table}[H]
	
	\caption{$C_3\wr S_1$}

    \begin{tabular}{|c|c|c|c|}
		
		\hline
		
	    $\gamma\backslash c$ & $c_1$ & $c_2$ & $c_3$\\
		
		\hline
		
		$\gamma_1$ & 1 & 1 & 1 \\
		
		\hline
		
		$\gamma_2$ & $\omega$ & $\omega^2$ & 1\\
		
		\hline
		
		$\gamma_3$ & $\omega^2$ & $\omega$ & 1\\
		
		\hline
		
	\end{tabular}

\end{table}
where $\gamma_1=c_3=((1),\emptyset,\emptyset)$; $\gamma_2=c_2=(\emptyset,(1),\emptyset)$; $\gamma_3=c_1=(\emptyset,\emptyset,(1))$.

\begin{table}[H]
	
	\caption{$C_3\wr S_2$}
	
	\begin{tabular}{|c|c|c|c|c|c|c|c|c|c|}
		
		\hline
		
		$\gamma\backslash c$ & $c_1$ & $c_2$ & $c_3$ & $c_4$ & $c_5$ & $c_6$ & $c_7$ & $c_8$ & $c_9$\\
		
		\hline
		
		$\gamma_1$ & 1 & 1 & 1 & 1 & 1 & 1 & 1 & 1 & 1\\
		
		\hline
		
		$\gamma_2$ & 1 & -1 & 1 & 1 & -1 & 1 & 1 & 1 & -1\\
		
		\hline
		
		$\gamma_3$ & 2$\omega$ & 0 & -1 & 2$\omega^2$ & 0 & -$\omega^2$ & -$\omega$ & 2 & 0\\
		
		\hline
		
		$\gamma_4$ & 2$\omega^2$ & 0 & -1 & 2$\omega$ & 0 & -$\omega$ & -$\omega^2$ & 2 & 0\\
		
		\hline
		
		$\gamma_5$ & $\omega^2$ & $\omega$ & 1 & $\omega$ & $\omega^2$ & $\omega$ & $\omega^2$ & 1 & 1\\
		
		\hline
		
		$\gamma_6$ & $\omega^2$ & -$\omega$ & 1 & $\omega$ & -$\omega^2$ & $\omega$ & $\omega^2$ & 1 & -1\\
		
		\hline
		
		$\gamma_7$ & 2 & 0 & -1 & 2 & 0 & -1 & -1 & 2 & 0\\
		
		\hline
		
		$\gamma_8$ & $\omega$ & $\omega^2$ & 1 & $\omega^2$ & $\omega$ & $\omega^2$ & $\omega$ & 1 & 1\\
		
		\hline
		
		$\gamma_9$ & $\omega$ & -$\omega^2$ & 1 & $\omega^2$ & -$\omega$ & $\omega^2$ & $\omega$ & 1 & -1\\
		
		\hline
		
	\end{tabular}

\end{table}
where $\gamma_1=c_9=((2),\emptyset,\emptyset)$; $\gamma_2=c_8=((1,1),\emptyset,\emptyset)$; $\gamma_3=c_7=((1),(1),\emptyset)$; $\gamma_4=c_6=((1),\emptyset,(1))$; $\gamma_5=c_5=(\emptyset,(2),\emptyset)$; $\gamma_6=c_4=(\emptyset,(1,1),\emptyset)$; $\gamma_7=c_3=(\emptyset,(1),(1))$; $\gamma_8=c_2=(\emptyset,\emptyset,(2))$; $\gamma_9=c_1=(\emptyset,\emptyset,(1,1))$.

\begin{table}[H]
	
	\caption{$C_3\wr S_3$}
	\renewcommand\arraystretch{1.5}
	\resizebox{\textwidth}{!}{
		\begin{tabular}{|c|c|c|c|c|c|c|c|c|c|c|c|c|c|c|c|c|c|c|c|c|c|c|}
			
			\hline
			
			$\gamma\backslash c$  & $c_1$ & $c_2$ & $c_3$ & $c_4$ & $c_5$ & $c_6$ & $c_7$ & $c_8$ & $c_9$ & $c_{10}$ & $c_{11}$ & $c_{12}$ & $c_{13}$ & $c_{14}$ & $c_{15}$ & $c_{16}$ & $c_{17}$ & $c_{18}$ & $c_{19}$ & $c_{20}$ & $c_{21}$ & $c_{22}$\\
			
			\hline
			
			$\gamma_1$ & 1 & 1 & 1 & 1 & 1 & 1 & 1 & 1 & 1 & 1 & 1 & 1 & 1 & 1 & 1 & 1 & 1 & 1 & 1 & 1 & 1 & 1\\
			
			\hline
			
			$\gamma_2$ & 2 & 0 & -1 & 2 & 0 & 2 & 0 & 2 & 0 & -1 & 2 & 0 & 2 & 2 & 0 & 2 & 0 & 2 & 0 & 2 & 0 & -1\\
			
			\hline
			
			$\gamma_3$ & 1 & -1 & 1 & 1 & -1 & 1 & -1 & 1 & -1 & 1 & 1 & -1 & 1 & 1 & -1 & 1 & -1 & 1 & -1 & 1 & -1 & 1\\
			
			\hline
			
			$\gamma_4$ & 3$\omega$ & $\omega$ & 0 & $\omega$-1 & $\omega^2$ & $\omega^2$-1 & $\omega$ & 3$\omega^2$ & $\omega^2$ & 0 & $\omega-\omega^2$ & 1 & 0 & $\omega^2-\omega$ & 1 & 1-$\omega^2$ & $\omega$ & 1-$\omega$ & $\omega^2$ & 3 & 1 & 0\\
			
			\hline
			
			$\gamma_5$ & 3$\omega$ & -$\omega$ & 0 & $\omega$-1 & -$\omega^2$ & $\omega^2$-1 & -$\omega$ & 3$\omega^2$ & -$\omega^2$ & 0 & $\omega-\omega^2$ & -1 & 0 & $\omega^2-\omega$ & -1 & 1-$\omega^2$ & -$\omega$ & 1-$\omega$ & -$\omega^2$ & 3 & -1 & 0\\
			
			\hline
			
			$\gamma_6$ & 3$\omega^2$ & $\omega^2$ & 0 & $\omega^2$-1 & $\omega$ & $\omega$-1 & $\omega^2$ & 3$\omega$ & $\omega$ & 0 & $\omega^2-\omega$ & 1 & 0 & $\omega-\omega^2$ & 1 & 1-$\omega$ & $\omega^2$ & 1-$\omega^2$ & $\omega$ & 3 & 1 & 0\\
			
			\hline
			
			$\gamma_7$ & 3$\omega^2$ & -$\omega^2$ & 0 & $\omega^2$-1 & -$\omega$ & $\omega$-1 & -$\omega^2$ & 3$\omega$ & -$\omega$ & 0 & $\omega^2-\omega$ & -1 & 0 & $\omega-\omega^2$ & -1 & 1-$\omega$ & -$\omega^2$ & 1-$\omega^2$ & -$\omega$ & 3 & -1 & 0\\
			
			\hline
			
			$\gamma_8$ & 3$\omega^2$ & $\omega$ & 0 & 1-$\omega$ & $\omega$ & 1-$\omega^2$ & $\omega^2$ & 3$\omega$ & $\omega^2$ & 0 & $\omega$-1 & $\omega$ & 0 & $\omega^2$-1 & $\omega^2$ & $\omega-\omega^2$ & 1 & $\omega^2-\omega$ & 1 & 3 & 1 & 0\\
			
			\hline
			
			$\gamma_9$ & 3$\omega^2$ & -$\omega$ & 0 & 1-$\omega$ & -$\omega$ & 1-$\omega^2$ & -$\omega^2$ & 3$\omega$ & -$\omega^2$ & 0 & $\omega$-1 & -$\omega$ & 0 & $\omega^2$-1 & -$\omega^2$ & $\omega-\omega^2$ & -1 & $\omega^2-\omega$ & -1 & 3 & -1 & 0\\
			
			\hline
			
			$\gamma_{10}$ & 6 & 0 & 0 & 0 & 0 & 0 & 0 & 6 & 0 & 0 & 0 & 0 & -3 & 0 & 0 & 0 & 0 & 0 & 0 & 6 & 0 & 0\\
			
			\hline
			
			$\gamma_{11}$ & 3$\omega$ & $\omega^2$ & 0 & 1-$\omega^2$ & $\omega^2$ & 1-$\omega$ & $\omega$ & 3$\omega^2$ & $\omega$ & 0 & $\omega^2$-1 & $\omega^2$ & 0 & $\omega$-1 & $\omega$ & $\omega^2-\omega$ & 1 & $\omega-\omega^2$ & 1 & 3 & 1 & 0\\
			
			\hline
			
			$\gamma_{12}$ & 3$\omega$ & -$\omega^2$ & 0 & 1-$\omega^2$ & -$\omega^2$ & 1-$\omega$ & -$\omega$ & 3$\omega^2$ & -$\omega$ & 0 & $\omega^2$-1 & -$\omega^2$ & 0 & $\omega$-1 & -$\omega$ & $\omega^2-\omega$ & -1 & $\omega-\omega^2$ & -1 & 3 & -1 & 0\\
			
			\hline
			
			$\gamma_{13}$ & 1 & $\omega^2$ & $\omega$ & $\omega$ & 1 & $\omega^2$ & 1 & 1 & $\omega$ & $\omega^2$ & $\omega^2$ & $\omega$ & 1 & $\omega$ & $\omega^2$ & $\omega$ & $\omega$ & $\omega^2$ & $\omega^2$ & 1 & 1 & 1\\
			
			\hline
			
			$\gamma_{14}$ & 2 & 0 & -$\omega$ & 2$\omega$ & 0 & 2$\omega^2$ & 0 & 2 & 0 & -$\omega^2$ & 2$\omega^2$ & 0 & 2 & 2$\omega$ & 0 & 2$\omega$ & 0 & 2$\omega^2$ & 0 & 2 & 0 & -1\\
			
			\hline
			
			$\gamma_{15}$ & 1 & -$\omega^2$ & $\omega$ & $\omega$ & -1 & $\omega^2$ & -1 & 1 & -$\omega$ & $\omega^2$ & $\omega^2$ & -$\omega$ & 1 & $\omega$ & -$\omega^2$ & $\omega$ & -$\omega$ & $\omega^2$ & -$\omega^2$ & 1 & -1 & 1\\
			
			\hline
			
			$\gamma_{16}$ & 3$\omega$ & 1 & 0 & $\omega^2-\omega$ & $\omega^2$ & $\omega-\omega^2$ & $\omega$ & 3$\omega^2$ & 1 & 0 & 1-$\omega$ & $\omega$ & 0 & 1-$\omega^2$ & $\omega^2$ & $\omega$-1 & $\omega^2$ & $\omega^2$-1 & $\omega$ & 3 & 1 & 0\\
			
			\hline
			
			$\gamma_{17}$ & 3$\omega$ & -1 & 0 & $\omega^2-\omega$ & -$\omega^2$ & $\omega-\omega^2$ & -$\omega$ & 3$\omega^2$ & -1 & 0 & 1-$\omega$ & -$\omega$ & 0 & 1-$\omega^2$ & -$\omega^2$ & $\omega$-1 & -$\omega^2$ & $\omega^2$-1 & -$\omega$ & 3 & -1 & 0\\
			
			\hline
			
			$\gamma_{18}$ & 3$\omega^2$ & 1 & 0 & $\omega-\omega^2$ & $\omega$ & $\omega^2-\omega$ & $\omega^2$ & 3$\omega$ & 1 & 0 & 1-$\omega^2$ & $\omega^2$ & 0 & 1-$\omega$ & $\omega$ & $\omega^2$-1 & $\omega$ & $\omega$-1 & $\omega^2$ & 3 & 1 & 0\\
			
			\hline
			
			$\gamma_{19}$ & 3$\omega^2$ & -1 & 0 & $\omega-\omega^2$ & -$\omega$ & $\omega^2-\omega$ & -$\omega^2$ & 3$\omega$ & -1 & 0 & 1-$\omega^2$ & -$\omega^2$ & 0 & 1-$\omega$ & -$\omega$ & $\omega^2$-1 & -$\omega$ & $\omega$-1 & -$\omega^2$ & 3 & -1 & 0\\
			
			\hline
			
			$\gamma_{20}$ & 1 & $\omega$ & $\omega^2$ & $\omega^2$ & 1 & $\omega$ & 1 & 1 & $\omega^2$ & $\omega$ & $\omega$ & $\omega^2$ & 1 & $\omega^2$ & $\omega$ & $\omega^2$ & $\omega^2$ & $\omega$ & $\omega$ & 1 & 1 & 1\\
			
			\hline
			
			$\gamma_{21}$ & 2 & 0 & -$\omega^2$ & 2$\omega^2$ & 0 & 2$\omega$ & 0 & 2 & 0 & -$\omega$ & 2$\omega$ & 0 & 2 & 2$\omega^2$ & 0 & 2$\omega^2$ & 0 & 2$\omega$ & 0 & 2 & 0 & -1\\
			
			\hline
			
			$\gamma_{22}$ & 1 & -$\omega$ & $\omega^2$ & $\omega^2$ & -1 & $\omega$ & -1 & 1 & -$\omega^2$ & $\omega$ & $\omega$ & -$\omega^2$ & 1 & $\omega^2$ & -$\omega$ & $\omega^2$ & -$\omega^2$ & $\omega$ & -$\omega$ & 1 & -1 & 1\\
			
			\hline
			\end{tabular}}
\end{table}
where $\gamma_1=c_{22}=((3),\emptyset,\emptyset)$; $\gamma_2=c_{21}=((2,1),\emptyset,\emptyset)$; $\gamma_3=c_{20}=((1,1,1),\emptyset,\emptyset)$; $\gamma_4=c_{19}=((2),(1),\emptyset)$; $\gamma_5=c_{18}=((1,1),(1),\emptyset)$; $\gamma_6=c_{17}=((2),\emptyset,(1))$; $\gamma_7=c_{16}=((1,1),\emptyset,(1))$; $\gamma_8=c_{15}=((1),(2),\emptyset)$; $\gamma_9=c_{14}=((1),(1,1),\emptyset)$; $\gamma_{10}=c_{13}=((1),(1),(1))$; $\gamma_{11}=c_{12}=((1),\emptyset,(2))$; $\gamma_{12}=c_{11}=((1),\emptyset,(1,1))$; $\gamma_{13}=c_{10}=(\emptyset,(3),\emptyset)$; $\gamma_{14}=c_9=(\emptyset,(2,1),\emptyset)$; $\gamma_{15}=c_8=(\emptyset,(1,1,1),\emptyset)$; $\gamma_{16}=c_7=(\emptyset,(2),(1))$; $\gamma_{17}=c_6=(\emptyset,(1,1),(1))$; $\gamma_{18}=c_5=(\emptyset,(1),(2))$; $\gamma_{19}=c_4=(\emptyset,(1),(1,1))$; $\gamma_{20}=c_3=(\emptyset,\emptyset,(3))$; $\gamma_{21}=c_2=(\emptyset,\emptyset,(2,1))$; $\gamma_{22}=c_1=(\emptyset,\emptyset,(1,1,1))$.
\bigskip

Moreover, by Theorem 4.4, one has
\begin{table}[H]
	
	\caption{$S_9$ mod 3}
		\begin{center}
			\begin{tabular}{|c|c|c|c|}

\hline

$\lambda\backslash c$  & $(9)$ & $(6,3)$ & $(3,3,3)$\\

\hline
$(7,1,1)$ & 1 & 1 & 1\\

\hline
$(4,1,1,1,1,1)$ & 2 & 0 & 2\\

\hline
$(1,1,1,1,1,1,1,1,1)$ & 1 & 2 & 1\\

\hline
$(4,3,2)$ &0&2&0\\

\hline
$(2,2,2,1,1,1)$&0 &2&0\\

\hline
$(4,4,1)$ & 0& 1& 0\\

\hline
$(3,2,1,1,1,1)$ & 0&1&0\\

\hline
$(5,2,2)$ & 0&1&0\\

\hline
$(2,2,2,2,1)$ &0 &1 &0\\

\hline
$(3,3,3)$ & 0& 0& 0\\

\hline
$(6,2,1)$ & 0& 2& 0\\

\hline
$(3,2,2,2)$ &0 &2& 0\\

\hline
$(8,1)$ &2 &2 &2\\

\hline
$(5,1,1,1,1)$& 1 & 0 & 1 \\

\hline
$(2,1,1,1,1,1,1,1)$ &2 &1 &2\\

\hline
$(5,4)$ &0 &2 &0\\

\hline
$(3,3,1,1,1)$ &0 &2 &0\\

\hline
$(6,3)$& 0& 1 &0\\
\hline
$(3,3,2,1)$& 0 &1 &0\\
\hline
$(9)$ & 1 &1 &1\\
\hline
$(6,1,1,1)$ &2 &0 &2\\
\hline
$(3,1,1,1,1,1,1)$ &1 &2 &1\\
\hline
	\end{tabular}
\end{center}
\end{table}
where $\lambda$ corresponds to $\gamma_1\sim \gamma_{22}$. They are consistient with the elements in the character table of $S_9$ mod 3.

\medskip

\bigskip
\centerline{\bf Acknowledgments}
\medskip
The work is supported in part by
the Simons Foundation grant no. MP-TSM-00002518 and the National Natural Science Foundation of China grant nos.
12171303.

\newpage
\appendix
\begin{section}{The source code for Theorem \ref{t:MN-gensym}}
\begin{lstlisting}
# Take k=3, n=5

P=PartitionTuples(3,5)
import random
lam=random.choice(P)
rho=random.choice(P)
print(lam,rho)

# Set environment

m=max(max(rho))
j=rho.index(max(rho))
r=Zmod(3)
w=SR.var('w')

import numpy
def ht(nums):
    return numpy.count_nonzero(nums)-1

# Filter out lam satisfying conditions

for k in range(3):
    if lam[k].size()>=m:
    tem=Partitions(lam[k].size()-m,outer=lam[k]).list()

    for i in tem:
        if SkewPartition([lam[k],i]).is_ribbon():
            coff =
            (-1)^ht(SkewPartition([lam[k],i]).row_lengths())*w^r(-j*k)
            l=list(lam)
            l[k]=i
            print(coff,"new lam:",PartitionTuple(l))
\end{lstlisting}
\end{section}

\begin{section}{Theorem \ref{t:3.9}}
\begin{lstlisting}
# Take k=3, n=5

P=PartitionTuples(3,5)
import random
lam=random.choice(P)
rho=random.choice(P)
print(lam,rho)

# Set environment

s=max(max(lam))
j=lam.index(max(lam))
r=Zmod(3)
w=SR.var('w')

def sub(nums):
    res = [[]]
    for num in nums:
    res += [ i + [num] for i in res]
    return res

import itertools
array1 = sub(rho[0])
array2 = sub(rho[1])
array3 = sub(rho[2])
combs = itertools.product(array1,array2,array3)

# Simulate the path

import copy
class Solution:
    def subsets(self, nums):
        result = []
        path = []
        self.backtracking(nums, 0, path, result)
        return result

    def backtracking(self, nums, startIndex, path, result):
        result.append(path[:])
        for i in range(startIndex, len(nums)):
            path.append(nums[i])
            self.backtracking(nums, i + 1, path, result)
            path.pop()

# Filter out rho satisfying conditions

for comb in combs:
    if sum(comb[0])+sum(comb[1])+sum(comb[2])<=5-s:
        box0 = comb[0]
        box1 = comb[1]
        box2 = comb[2]
        taus=
        Partitions(5-s-sum(comb[0])-sum(comb[1])-sum(comb[2])).list()
        solu=Solution()
        result = {}

        for tau in taus:
            d=Partition(tau).centralizer_size()
            c=(-1/3)^len(tau)/d*w^r(-(len(rho[1])-
            len(comb[1]))*1*j-(len(rho[2])-len(comb[2]))*2*j)
            all_index_3 = list(range(len(tau)))
            all_index_1 = solu.subsets(all_index_3)

            for i_i in range(len(all_index_1)):
                index_2_3 = list(set(all_index_3)-set(all_index_1[i_i]))
                all_index_2_3 = solu.subsets(index_2_3)
                box0_i_index = copy.deepcopy(all_index_2_3)
                box1_i_index = copy.deepcopy(all_index_2_3)
                box2_i_index = copy.deepcopy(all_index_2_3)

                for i_i_i in range(len(all_index_2_3)):
                    box0_i_index[i_i_i] = list(set(all_index_3)-set(index_2_3))
                    box1_i_index[i_i_i] = list(set(index_2_3)-set(all_index_2_3[i_i_i]))

                    for k_0 in range(len(box0_i_index[i_i_i])):
                    box0_i_index[i_i_i][k_0] = tau[box0_i_index[i_i_i][k_0]]
                    for k_1 in range(len(box1_i_index[i_i_i])):
                    box1_i_index[i_i_i][k_1] = tau[box1_i_index[i_i_i][k_1]]
                    for k_2 in range(len(box2_i_index[i_i_i])):
                    box2_i_index[i_i_i][k_2] = tau[box2_i_index[i_i_i][k_2]]

                    coff =c*( w^r((len(box0_i_index[i_i_i]) * 0 + len(box1_i_index[i_i_i]) * 1 + len(box2_i_index[i_i_i]) * 2)*j))
                    x = box0 + box0_i_index[i_i_i]
                    y = box1 + box1_i_index[i_i_i]
                    z = box2 + box2_i_index[i_i_i]
                    tmp = list()
                    tmp.append(x)
                    tmp.append(y)
                    tmp.append(z)
                    result["rho"] = tmp
                    print(coff,"new rho:",result["rho"])
\end{lstlisting}
\end{section}
\end{document}